\documentclass[a4]{amsart}

\input xypic 
\input xy 
\xyoption{all}

\usepackage[T1]{fontenc}

\usepackage{amssymb} 
\usepackage{bbm}
\usepackage{tikz-cd}
\usepackage[all]{xy}    
\usepackage{graphicx}
\usepackage{mathrsfs}
\usepackage{multirow}
\usepackage{float}
\usepackage{adjustbox}
\usepackage{amsthm}
\usepackage{accents}
\usepackage{mathtools}
\usepackage{tikz}
\usepackage{caption}

\usepackage[bookmarksnumbered, bookmarksopen,
colorlinks,citecolor=blue,linkcolor=blue,backref]{hyperref}

\newtheorem{theorem}{Theorem}[section]
\newtheorem{proposition}[theorem]{Proposition}
\newtheorem{lemma}[theorem]{Lemma}
\newtheorem{corollary}[theorem]{Corollary}

\theoremstyle{definition}

\newcounter{bean}

\newcommand{\qqed}{\hfill\Box}

\begin{document}


\title[Homotopy of $6$-manifolds]
   {Loop homotopy of $6$-manifolds over $4$-manifolds} 

\author{Ruizhi Huang} 
\address{Institute of Mathematics, Academy of Mathematics and Systems Science, 
   Chinese Academy of Sciences, Beijing 100190, China} 
\email{haungrz@amss.ac.cn} 
   \urladdr{https://sites.google.com/site/hrzsea/}



\begin{abstract} 
Let $M$ be the $6$-manifold $M$ arising as the total space of the sphere bundle of a rank $3$ vector bundle over a simply connected closed $4$-manifold. We show that after looping $M$ is homotopy equivalent to a product of loops on spheres in general. This particularly implies a cohomology rigidity property of $M$ after looping. Furthermore, passing to the rational homotopy, we show that such $M$ is Koszul in the sense of Berglund.
\end{abstract}

\maketitle


\section{Introduction}
Classification of manifolds is a fundamental problem in geometry and topology.
There are tremendous investigations around this problem in both the smooth and topological category.
For instance, in the general case, Wall \cite{Wal1, Wal3} studied $(n-1)$-connected $2n$-manifolds and $(n-1)$-connected $(2n+1)$-manifolds. For concrete cases with specified dimension, Bardon \cite{Bar} classified simply connected $5$-manifolds, and Wall \cite{Wal2}, Jupp \cite{Jup} and Zhubr \cite{Zhu1, Zhu2} classified simply connected $6$-manifolds. More recently, Kreck and Su \cite{KS} classified certain non-simply connected $5$-manifolds, while Crowley and Nordstr\"{o}m \cite{CN} and Kreck \cite{Kre} studied the classification of various kinds of $7$-manifolds.

In the literature mentioned, the homotopy classification of $M$ was usually carried out as a byproduct in terms of a system of invariants. However, it is almost impossible to extract nontrivial homotopy information of $M$ directly from the classification. On the other hand, unstable homotopy theory is a powerful tool for studying the homotopy properties of manifolds preserved by suspending or looping. From the suspension aspect, So and Theriault \cite{ST} determined the homotopy type of the suspension of connected $4$-manifolds, while Huang \cite{H} studied the suspension of simply connected $6$-manifolds. From the loop aspect, Beben and Theriualt \cite{BT1} studied the loop decompositions of $(n-1)$-connected $2n$-manifolds, while Beben and Wu \cite{BW} and Huang and Theriault \cite{HT} studied the loop decompositions of the $(n-1)$-connected $(2n+1)$-manifolds. The homotopy groups of these manifolds were also investigated by Sa. Basu and So. Basu \cite{BB, Bas} from different point of view. Moreover, a theoretical method of loop decomposition was developed by Beben and Theriault \cite{BT2}, which is quite useful for studying the homotopy of manifolds. 

In this paper, we study the loop homotopy of certain simply connected $6$-manifolds arising from $4$-manifolds.
Let $N$ be a simply connected closed $4$-manifold such that $H^2(N;\mathbb{Z})\cong \mathbb{Z}^{\oplus d}$ with $d\geq 1$. A rank $3$ vector bundle $\xi$ over $N$ is classified by a map $f: N\longrightarrow BSO(3)$, where $BSO(3)$ is the classifying space of the special orthogonal group $SO(3)$. The sphere bundle of $\xi$
\begin{equation}\label{Mdefeq}
S^2\stackrel{i}{\longrightarrow} M\stackrel{p}{\longrightarrow} N
\end{equation}
defines the closed $6$-manifold $M$. Since the integral cohomologies of $N$ and $S^2$ are free and concentrated in even degree, the Serre spectral sequence of (\ref{Mdefeq}) collapses and $
H^\ast(M;\mathbb{Z})\cong H^\ast(N;\mathbb{Z})\otimes H^\ast(S^2;\mathbb{Z})$.
Our main result is the following theorem which will be proved in Section \ref{sec: pf}.

\begin{theorem}\label{decomthm}
Let $N$ be a simply connected closed $4$-manifold such that $H^2(N;\mathbb{Z})\cong \mathbb{Z}^{\oplus d}$ with $d\geq 1$. Let $M$ be the total manifold of the sphere bundle of any rank $3$ vector bundle over $N$. Then
\begin{itemize}
\item if $d=1$, 
   \[ \Omega M\simeq S^1\times \Omega S^2\times \Omega S^5,\] 
\item if $d\geq 2$,
   \[
   \Omega M\simeq S^1\times \Omega S^2\times\Omega (S^2\times S^3)\times  \Omega\big(J\vee(J\wedge\Omega (S^2\times S^3))\big),
   \]
where $J=\mathop{\bigvee}\limits_{i=1}^{d-2}(S^2\vee S^3)$.
   \end{itemize}
\end{theorem}

From Theorem \ref{decomthm} and its proof, it can be easily seem that the decomposition in Theorem \ref{decomthm} is compatible with the $S^2$-bundle (\ref{Mdefeq}) after looping. In particular, this means that though the fibre bundle (\ref{Mdefeq}) does not split in general, its loop does.
Moreover, as discussed in \cite[Page 217]{BT1}, the term $J\vee(J\wedge\Omega (S^2\times S^3))$ in the decomposition of Theorem \ref{decomthm} is a bouquet of spheres. Hence by the Hilton-Milnor theorem, we see that $\Omega M$ is homotopy equivalent to a product of loops on spheres with $S^1$. Additionally since the decompositions of Theorem \ref{decomthm} only depend on the value of $d$ which is determined by and determines $H^2(M;\mathbb{Z})$, we have the rigidity property of $M$ after looping.
\begin{corollary}
Let $M$ and $M^\prime$ be two $6$-manifolds in Theorem \ref{decomthm}. Then $\Omega M\simeq \Omega M^\prime$ if and only if $H^2(M;\mathbb{Z})\cong H^2(M^\prime;\mathbb{Z})$. ~$\qqed$
\end{corollary}

Theorem \ref{decomthm} can be improved if we pass from integral homotopy to rational homotopy. Indeed, by Theorem \ref{decomthm} it is straightforward to compute the homotopy groups of $M$ in terms of those of spheres. However, there is an additional Lie algebra structure on the homotopy groups of any $CW$ complex $X$. In rational homotopy theory, the graded Lie algebra $\pi_\ast(\Omega X)\otimes \mathbb{Q}$ is called the {\it homotopy Lie algebra} of $X$, and $X$ is called {\it coformal} if the rational homotopy type of $X$ is completely determined by its homotopy Lie algebra. If $X$ is further {\it formal}, that is the homotopy type of $X$ is determined by the graded commutative algebra $H^\ast(X;\mathbb{Q})$, then $X$ is {\it Koszul} in the sense of Berglund \cite[Definition 1.1]{Ber}. In the latter case, $H^\ast(X;\mathbb{Q})$ is a {\it Koszul algebra} and $\pi_\ast(\Omega X)\otimes\mathbb{Q}$ is a {\it Koszul Lie algebra} \cite{Ber}. The following theorem concerns these additional structures on $M$ in Theorem \ref{decomthm}.
\begin{theorem}\label{coformalthm}
Let $N$ be a simply connected closed $4$-manifold such that $H^2(N;\mathbb{Z})\cong \mathbb{Z}^{\oplus d}$. Let $M$ be the total manifold of the sphere bundle of any rank $3$ vector bundle over $N$. Then 
\begin{itemize}
\item if $d=1$, $M$ is not coformal, 
\item if $d\geq 2$, $M$ is Koszul, and there is an isomorphism of graded Lie algebras
\[
\pi_\ast(\Omega M)\otimes\mathbb{Q}\cong H^\ast(M;\mathbb{Q})^{!\mathscr{L}ie},
\]
where $(-)^{!\mathscr{L}ie}$ is the Koszul dual Lie functor defined in \cite[Section 2]{Ber}.
\end{itemize}
\end{theorem}

We turn to the remaining case when $d=0$, that is $N\cong S^4$. Note we still have the $6$-manifold $M$ as constructed in (\ref{Mdefeq}). Though the homotopy classification of such manifolds was almost determined by Yamaguchi \cite{Yam}, this case is surprisingly much harder than the general one. We will explain this point after the statement of our result in this case proved in Proposition \ref{Md=0koddprop} and Proposition \ref{Md=0kevenprop}. Let $\eta_2: S^3\rightarrow S^2$ be the Hopf map.
For any integer $n$, let $S^{m}\{n\}$ be the homotopy fibre of the degree $n$ map on $S^{m}$.
\begin{theorem}\label{d=0thm}
Let $M$ be the total space of the sphere bundle of any rank $3$ vector bundle over $S^4$. Then $M$ has a cell structure of the form
\[
M\simeq S^2\cup_{k\eta_2}e^4\cup e^6,
\]
where $k\in \mathbb{Z}$. Let $k=p_1^{r_1}\cdots p_\ell^{r_\ell}$ be the prime decomposition of $k$. 
We further have
\begin{itemize}
\item if $k$ is odd, 
\[
\Omega M\simeq S^1\times \prod_{j=1}^{\ell} S^3\{p_j^{r_j}\}\times \Omega S^7,
\]
\item if $k=2^r$ with $r\geq 3$,
\[
\Omega M \simeq S^1\times S^3\{2^r\}\times \Omega S^7.
\]
\end{itemize}
\end{theorem}
Note that we still have cohomology rigidity in this case since the homotopy type of $\Omega M$ only depends $k$ which is determined by the square of a generator in $H^2(M;\mathbb{Z})$. But it is less interesting since the rigidity of $M$ without looping holds except for the case when $k$ is even and $M$ is Spin \cite{Yam}.
Also note that Theorem \ref{d=0thm} is only a partial result.
To explain the difficulty in this case, remark that the proof of Theorem \ref{d=0thm} heavily relies on the result of \cite{HT} on the loop decomposition of $2$-connected $7$-manifolds.
As discussed in \cite[Section 6]{HT}, the case when $k=2^r m$ with $m$ odd and greater than $1$ is much more difficult. Also, since it is known that $S^3\{2\}$ is not an $H$-space \cite{C}, we can not have a decomposition of the form $\Omega M \simeq S^1\times S^3\{2\}\times \Omega S^7$ for the case when $k=2$. In contrast, the rational homotopy of $M$ in this case is simple. As showed in Lemma \ref{d=0rationallemma}, $M$ is rationally homotopic equivalent to $\mathbb{C}P^3$ or $S^2\times S^4$. Moreover, it is well known that $\mathbb{C}P^3$ is not coformal \cite[Example 4.7]{NM}, while $S^2\times S^4$ is Koszul \cite[Example 5.1 and 5.4]{Ber}.

Before we close the Introduction, let us make two remarks. Firstly, our results provide further evidence on the Moore conjecture. Recall that the Moore conjecture states that a simply connected finite $CW$ complex $Z$ is rationally elliptic if and only if it has a finite homotopy exponent at all primes, or equivalently, $Z$ is rationally hyperbolic if and only if it has unbound homotopy exponent at some prime. For $M$ in our context, it is elliptic if and only if when $d\leq 2$, and in any of these cases by \cite{J, CMN, CMN2} $M$ has a finite homotopy exponent at all primes. When $d\geq 3$, $M$ is hyperbolic so that $\Omega M$ has $\Omega(S^2\vee S^3)$ as product summand, hence it has no bound homotopy exponent for any prime $p$ (see \cite{Boy} for instance). Secondly, Amor\'{o}s and Biswas \cite{AB} characterized simply connected rationally elliptic compact K\"{a}hler threefolds in terms of Hodge diamonds, and in particular, their second Betti numbers $b_2\leq 3$. For $M$ in our context, this is equivalent to $d\leq 2$, and our decompositions provide further information on the homotopy of $M$. For instance the homotopy groups of $M$ can be computed in terms of those of spheres.

The paper is organized as follows. In Section \ref{sec: prelim} we classify rank $3$ bundles over the $4$-manifold $N$. In Section \ref{sec: null}, we prove Lemma \ref{nulllemma} which implies that under Lemma \ref{Nbundlelemma} one component of the classifying map $f$ of the bundle $\xi$ over $N$ is trivial in a special case. This is crucial for proving Theorem \ref{decomthm}. In Section \ref{sec: pf}, we prove Theorem \ref{decomthm} by dividing it into two cases. Section \ref{sec: d=0} is devoted to the remaining case when $d=0$ and we prove Theorem \ref{d=0thm} there. We discuss the rational homotopy of $6$-manifolds and prove Theorem \ref{coformalthm} in Section \ref{sec: rational}.

\bigskip

\noindent{\bf Acknowledgements.}
Ruizhi Huang was supported by National Natural Science Foundation of China (Grant nos. 11801544 and 11688101), and ``Chen Jingrun'' Future Star Program of AMSS. He would like to thank Professor Stephen Theriault for the international online lecture series ``Loop Space Decomposition'', which stimulated his research interest in the homotopy of $6$-manifolds. He also want to thank Professor Yang Su for helpful discussions on obstructions to trivializing vector bundles.

\section{Rank $3$ bundles over $4$-manifolds} 
\label{sec: prelim}
In this section, we discuss necessary knowledge of rank $3$ vector bundles over simply connected $4$-manifolds, which will be used in the subsequent sections. There are various ways to study the classification of vector bundles. Here, we adopt an approach from homotopy theoretical point of view for the latter use.

Let $N$ be a simply connected $4$-manifold such that $H^2(N;\mathbb{Z})\cong \mathbb{Z}^{\oplus d}$ with $d\geq 0$. A rank $3$ vector bundle $\xi$ over $N$ is classified by a map $f: N\longrightarrow BSO(3)$. The sphere bundle of $\xi$
\[
S^2\stackrel{i}{\longrightarrow} M\stackrel{p}{\longrightarrow} N
\]
defines the closed $6$-manifold $M$. 
For $N$ there is the homotopy cofiber sequence
\begin{equation}\label{Ncofibreeq}
S^3\stackrel{\phi}{\longrightarrow} \bigvee_{i=1}^{d} S^2 \stackrel{\rho}{\longrightarrow} N \stackrel{q}{\longrightarrow} S^4\stackrel{\Sigma\phi}{\longrightarrow} \bigvee_{i=1}^{d} S^3,
\end{equation}
where $\phi$ is the attaching map of the top cell of $N$, $\rho$ is the injection of the $2$-skeleton, and $q$ is the pinch map onto the top cell. Let $s: S^1\cong SO(2)\rightarrow SO(3)$ be the canonical inclusion of Lie groups.

\begin{lemma}\label{Nbundlelemma}
There is a surjection 
\[
\Phi: [S^4, BSO(3)]\times [N, BS^1]\longrightarrow [N, BSO(3)]
\]
of pointed sets such that it restricts to $q^\ast$ on $[S^4, BSO(3)]$, and to $(Bs)_\ast$ on $[N, BS^1]$.
\end{lemma}
\begin{proof}
By (\ref{Ncofibreeq}), there is the exact sequence of pointed sets
\[
0=[\mathop{\bigvee}\limits_{i=1}^{d} S^3, BSO(3)] \stackrel{}{\longrightarrow}[S^4, BSO(3)] \stackrel{q^\ast}{\longrightarrow} [N, BSO(3)] \stackrel{\rho^\ast}{\longrightarrow} [\bigvee_{i=1}^{d} S^2, BSO(3)] \stackrel{}{\longrightarrow} [S^3, BSO(3)]=0,
\]
in a strong sense that, there is an action of $[S^4, BSO(3)]$ on $[N, BSO(3)]$ through $q^\ast$ such that the sets $\rho^{\ast-1} (x)$, for $x\in [\mathop{\bigvee}\limits_{i=1}^{d} S^2, BSO(3)]$ are precisely the orbits.
It is known that $[\mathop{\bigvee}\limits_{i=1}^{d} S^2, BSO(3)]\cong \mathop{\oplus}\limits_{d}\mathbb{Z}/2\mathbb{Z}$ and $[S^4, BSO(3)]\cong \mathbb{Z}$. 
Moreover, there is the commutative diagram
\[
\xymatrix{
[N, BS^1] \ar[r]^<<<<<{\rho^\ast}_<<<<<{\cong} \ar[d]^{(Bs)_\ast}& 
[\mathop{\bigvee}\limits_{i=1}^{d} S^2, BS^1] \ar[d]^{(Bs)_\ast} \ar[r]^<<<<<{\cong} 
& \mathop{\oplus}\limits_{d}\mathbb{Z} \ar[d]^{\mathop{\oplus}\limits_{d}\rho_2} \\
[N, BSO(3)] \ar[r]^<<<<{\rho^\ast}  
& [\mathop{\bigvee}\limits_{i=1}^{d} S^2, BSO(3)]\ar[r]^<<<{\cong} 
& \mathop{\oplus}\limits_{d}\mathbb{Z}/2\mathbb{Z},
}
\]
where $\rho^\ast$ is an isomorphism onto $[\mathop{\bigvee}\limits_{i=1}^{d} S^2, BS^1]\cong \mathop{\oplus}\limits_{d}\mathbb{Z}$, $\rho_2$ is the mod-$2$ reduction, and hence $(Bs)_\ast$ is surjective onto $[\mathop{\bigvee}\limits_{i=1}^{d} S^2, BSO(3)]$. Now for any $f\in [N, BSO(3)]$, $\rho^\ast(f)=(Bs)_\ast(x)$ for some $x\in [\mathop{\bigvee}\limits_{i=1}^{d} S^2, BS^1]$. Denote $\alpha=(\rho^{\ast-1})(x)$, then $Bs_\ast(\alpha)$ and $f$ belong to the same orbit of the action, for they have same image in $[\mathop{\bigvee}\limits_{i=1}^{d} S^2, BSO(3)]$ through $\rho^\ast$. Hence, there exists a $f^\prime\in [S^4, BSO(3)]$ such that $q^\ast(f^\prime)\cdot (Bs_\ast(\alpha))=f$.
This completes the proof of the lemma.
\end{proof}

From Lemma \ref{Nbundlelemma} and its proof, for the classifying map $f: N\rightarrow BSO(3)$, we have associated a pair of maps 
\begin{equation}\label{Nbundleeq}
(f^\prime, \alpha)\in [S^4, BSO(3)]\times [N, BS^1] ~{\rm such}~{\rm that}~q^\ast(f^\prime)\cdot (Bs_\ast(\alpha))=f, ~\ \omega_2(\xi)\equiv \alpha~{\rm mod}~2.
\end{equation} 
We also notice that if $\rho^\ast(f)\neq0$, or equivalently, $\xi$ is non-Spin, the element $\alpha$ can be always chosen to be primitive, that is, $\alpha$ is not divisible by any integer $k$ with $k\neq \pm 1$. This is important for our later use.

Let $\pi: W\rightarrow N$ be a map from a closed manifold $W$. The pullback of the bundle $\xi$ along $\pi$ has an associated sphere bundle
\[
S^2\stackrel{\iota}{\longrightarrow} Z\stackrel{\mathfrak{p}}{\longrightarrow} W
\]
which defines the closed manifold $Z$. The following lemma is critical for proving Proposition \ref{pdtbundleprop}.
\begin{lemma}\label{pdtbundlelemma}
Suppose for $W$ there is a homotopy cofibration
\[
W_{m-1}\stackrel{\varrho}{\longrightarrow} W\stackrel{\mathfrak{q}}{\longrightarrow} S^{m},
\]
such that $\pi\circ \varrho$ factors as
\[
W_{m-1}\stackrel{\pi_\prime}{\longrightarrow}\bigvee_{i=1}^{d} S^2 \stackrel{\rho}{\longrightarrow} N
\]
for some $\pi_\prime$, where $W_{m-1}$ is the $(m-1)$-skeleton of $W$.
Then if $f^\prime \circ q \circ \pi$ and $\alpha \circ \pi$ are both null homotopic, the bundle $\pi^\ast(\xi)$ is trivial, and in particular
\[
Z\cong S^2\times W.
\]
\end{lemma}
\begin{proof}
By the assumption, there is a diagram of homotopy cofibrations
\[
\xymatrix{
W_{m-1} \ar[r]^{\varrho} \ar[d]^{\pi_\prime}  &
W \ar[r]^{\mathfrak{q}} \ar[d]^{\pi} &
S^m \ar[d]^{\pi^\prime} \\
\mathop{\bigvee}\limits_{d} S^2 \ar[r]^{\rho}&
N \ar[r]^{q} &
S^4,
}
\]
which defines the map $\pi^\prime$. It follows that there is a morphism of exact sequences of pointed sets
\[
\xymatrix{
[S^4, BSO(3)]  \ar[r]^{q^\ast}  \ar[d]^{\pi^{\prime\ast}}&
[N, BSO(3)] \ar[r]^<<<<{\rho^\ast} \ar[d]^{\pi^\ast}&
[\mathop{\bigvee}\limits_{i=1}^{d} S^2, BSO(3)] \ar[d]^{\pi_\prime^\ast} \\
[S^m, BSO(3)]  \ar[r]^{\mathfrak{q}^\ast} &
[W, BSO(3)] \ar[r]^<<<<{\varrho^\ast} &
[W_{m-1}, BSO(3)], 
}
\]
such that the action of $[S^4, BSO(3)]$ on $[N, BSO(3)]$ is compatible with that of $[S^m, BSO(3)]$ on $[W, BSO(3)]$ through $\pi^{\prime\ast}$. Hence by (\ref{Nbundleeq}) the classifying map $f\circ\pi$ of $\pi^\ast(\xi)$ satisfies
\[
\begin{split}
f\circ\pi
&=\pi^\ast\big(q^\ast(f^\prime)\cdot (Bs_\ast(\alpha))\big)\\
&=\mathfrak{q}^\ast(\pi^{\prime\ast}(f^\prime))\cdot \pi^\ast(Bs_\ast(\alpha))\\ 
&=\pi^\ast(q^\ast(f^\prime))\cdot\pi^\ast((Bs_\ast(\alpha)))\\
&= (f^\prime \circ q \circ \pi)\cdot Bs_\ast(\alpha \circ \pi),
\end{split}
\]
which is null homotopic by the assumption. The lemma then follows immediately. 
\end{proof}

Lemma \ref{Nbundlelemma} also gives a byproduct on the classification of rank $3$ vector bundles over $N$ via characteristic classes, which could be also proved by other methods like the classical obstruction theory.
\begin{proposition}\label{classbundleNprop}
A rank $3$ vector bundle $\xi$ over $N$ is completely determined by its second Stiefel-Whitney class $\omega_2(\xi)$ and its first Pontryagin class $p_1(\xi)$.
\end{proposition}
\begin{proof}
Given any two rank $3$ vector bundles $\xi_1$ and $\xi_2$ over $N$. Suppose that $\omega_2(\xi_1)=\omega_2(\xi_2)$ and $p_1(\xi_1)=p_1(\xi_2)$. We want to show that $\xi_1\cong \xi_2$, or equivalently, $f_1\simeq f_2$, where $f_1$ and $f_2: N\rightarrow BSO(3)$ are the classifying maps of $\xi_1$ and $\xi_2$ respectively. By Lemma \ref{Nbundlelemma} and (\ref{Nbundleeq}), we have $f_1=q^\ast(f_1^\prime)\cdot (Bs_\ast(\alpha))$ for a pair of maps $(f_1^\prime, \alpha)\in [S^4, BSO(3)]\times [N, BS^1]$ such that $\omega_2(\xi_1)\equiv \alpha~{\rm mod}~2$. Since $\omega_2(\xi_1)=\omega_2(\xi_2)$, there exists some $f_2^\prime\in [S^4, BSO(3)]$ such that $f_2=q^\ast(f_2^\prime)\cdot (Bs_\ast(\alpha))$. It follows that, to show $f_1\simeq f_2$, it suffices to show $f_1^\prime\simeq f_2^\prime$. Indeed, for either $\xi_i$ the expression of $f_i$ can be explicitly described as
\[
f_i: N\stackrel{\mu^\prime}{\longrightarrow} N\vee S^4\stackrel{\alpha\vee f_i^\prime}{\longrightarrow} BS^1\vee BSO(3)\stackrel{Bs\vee {\rm id}}{\longrightarrow}BSO(3)\vee BSO(3) \stackrel{\nabla}{\longrightarrow}BSO(3),
\]
where $\mu^\prime$ is the co-action map, $\nabla$ is the folding map. In particular, it is easy to see that
\begin{equation}\label{p1xiieq}
p_1(\xi_i)=q^\ast(p_1(f_i^\prime))+\alpha^2,
\end{equation}
where we denote $p_1(f_i^\prime)$ to be the first Pontryagin class of the bundle over $S^4$ determined by $f_i^\prime$. Since $p_1(\xi_1)=p_1(\xi_2)$, (\ref{p1xiieq}) implies that $q^\ast(p_1(f_1^\prime))=q^\ast(p_2(f_i^\prime))$. Moreover, it is clear that $q^\ast: H^4(S^4;\mathbb{Z})\rightarrow H^4(N;\mathbb{Z})$ is an isomorphism. Hence $p_1(f_1^\prime)=p_1(f_2^\prime)$. 
Now since $[S^4, BSO(3)]\simeq \mathbb{Z}$ and the morphism $\frac{p_1}{4}:[S^4, BSO(3)]\rightarrow H^4(S^4;\mathbb{Z})$, sending each map to one fourth of the first Pontryagin class of the associated bundle, is an isomorphism \cite{HBJ}, we see that $f_1^\prime\simeq f_2^\prime$. Then $f_1\simeq f_2$ and the proposition follows.
\end{proof}

\section{The induced map between top cells} 
\label{sec: null} 
Let $N$ be a simply connected closed $4$-manifold such that $H^2(N;\mathbb{Z})\cong \mathbb{Z}^{\oplus d}$ with $d\geq 1$. Consider the circle bundle
\[
S^1\stackrel{j}{\longrightarrow} Y\stackrel{\pi}{\longrightarrow} N
\]
classified by a primitive element $\beta\in H^2(N;\mathbb{Z})$, which defines the simply connected $5$-manifold $Y$. 
By \cite[Lemma 1]{DL}, $Y$ has cell structure of the form
\[
Y\simeq  \mathop{\bigvee}\limits_{d-1} (S^2\vee S^3)\cup e^5. 
\]
Then by the cellular approximation theorem, there is the diagram of homotopy cofibration 
\begin{equation}\label{piprimediag}
\begin{gathered}
 \xymatrix{
\mathop{\bigvee}\limits_{d-1} (S^2\vee S^3) \ar[d] \ar[r]^<<<{\varrho} &
Y  \ar[d]^{\pi} \ar[r]^>>>>{\mathfrak{q}} &
S^5 \ar[d]^{\pi^\prime} \\
\mathop{\bigvee}\limits_{d} S^2 \ar[r]^{\rho} &
N\ar[r]^{q} &
S^4,
}
\end{gathered}
\end{equation}
where the bottom cofibration is part of (\ref{Ncofibreeq}), $\varrho$ is the inclusion of the $3$-skeleton of $Y$ followed by the quotient $\mathfrak{q}$, $\pi^\prime$ is induced from $\pi$. In this section, we prove the following key lemma for understanding rank $3$-bundles over $Y$ in a special case. 
Let $[N]$ be the fundamental class of $N$. Let $\langle x\cup y, [N] \rangle\in \mathbb{Z}$ be the canonical pairing for any cohomology classes $x$, $y\in H^2(N;\mathbb{Z})$.

\begin{lemma}\label{nulllemma}
The induced map $\pi^\prime$ in Diagram (\ref{piprimediag}) is null homotopic when $\langle \beta^2, [N]\rangle$ is odd.
\end{lemma}
\begin{proof}
The primitive element $\beta$ is represented by a map $\beta: N\rightarrow \mathbb{C}P^\infty\simeq K(\mathbb{Z},2)$. By the cellular approximation theorem $\beta$ factors through $\mathbb{C}P^2$
\[
\beta: N\stackrel{\widetilde{\beta}}{\longrightarrow}\mathbb{C}P^2\stackrel{x}{\longrightarrow}\mathbb{C}P^\infty,
\]
which defines the map $\widetilde{\beta}$, and $x$ represents a generator $x\in H^2(\mathbb{C}P^2;\mathbb{Z})$. The factorization gives a diagram of circle bundles
\begin{equation}
\label{NCP2bundleeq}
\begin{gathered}
 \xymatrix{
S^1 \ar@{=}[d] \ar[r]^{j} &
Y \ar[r]^{\pi} \ar[d]^{\widehat{\beta}}&
N \ar[d]^{\widetilde{\beta}}\\
S^1 \ar[r]^{} &
S^5 \ar[r]^{\pi_0} &
\mathbb{C}P^2,
}
\end{gathered}
\end{equation}
where the bundle in the second row is classified by $x$, and $\widehat{\beta}$ is the induced map. By the cellular approximation theorem, there is a homotopy commutative diagram
\begin{gather}
\begin{aligned}
\xymatrix{
 & Y \ar[dl]_{\mathfrak{q}}  \ar[rr]^{\pi}  \ar[dd]^<<<<{\widehat{\beta}}|!{[d];[d]}\hole && 
 N \ar[dl]_{q}  \ar[dd]^{\widetilde{\beta}}  \\
S^5\ar[dd]^{\widehat{\beta}^\prime}  \ar[rr]^<<<<<{\pi^\prime}  &&
S^4 \ar[dd]^<<<<{\widetilde{\beta}^\prime} \\
  & S^5 \ar@{=}[dl]  \ar[rr]^<<<<<{\pi_0}|!{[r];[r]}\hole  && 
 \mathbb{C}P^2\ar[dl]_{q_0}    \\
 S^5   \ar[rr]^{\pi_0^\prime}  &&
 S^4,  
}
\end{aligned}
\label{oddlemmadiag}
\end{gather}
where the rear and top faces are the right squares in Diagram (\ref{NCP2bundleeq}) and Diagram (\ref{piprimediag}) respectively, $q_0$ is the quotient map onto the top cell of $\mathbb{C}P^2$, $\pi_0^\prime$ is defined to be $q_0\circ\pi_0$, and $\widehat{\beta}^\prime$ and $\widetilde{\beta}^\prime$ are the induced maps. By the homotopy commutativity of the right face of Diagram (\ref{oddlemmadiag}), the assumption that $\langle \beta^2, [N]\rangle $ is odd is equivalent to $\widetilde{\beta}^\prime$ being odd degree.
Further, since the homotopy cofibre of $\pi_0$ is $\mathbb{C}P^3$ for which the Steenrod operation $Sq^2: H^4(\mathbb{C}P^3;\mathbb{Z}/2\mathbb{Z})\rightarrow H^6(\mathbb{C}P^3;\mathbb{Z}/2\mathbb{Z})$ is trivial, we obtain that $\pi_0^\prime=q_0\circ\pi_0$ is null homotopic.
Now consider the front face of Diagram (\ref{oddlemmadiag}). Combining the above arguments and the fact that $\pi_5(S^4)\cong \mathbb{Z}/2\mathbb{Z}\{\eta_4\}$ \cite{Tod} we see that
$
\pi^\prime \simeq \widetilde{\beta}^\prime\circ\pi^\prime\simeq \pi_0^\prime\circ \widehat{\beta}^\prime 
$
is null homotopic. This proves the lemma.
\end{proof}

\section{Proof of Theorem \ref{decomthm}} 
\label{sec: pf} 
Let $N$ be a simply connected $4$-manifold such that $H^2(N;\mathbb{Z})\cong \mathbb{Z}^{\oplus d}$ with $d\geq 1$. A rank $3$ vector bundle $\xi$ over $N$ is classified by a map $f: N\longrightarrow BSO(3)$ with the associated sphere bundle
\[
S^2\stackrel{i}{\longrightarrow} M\stackrel{p}{\longrightarrow} N,
\]
which defines the closed $6$-manifold $M$. 
Recall by Lemma \ref{Nbundlelemma} and (\ref{Nbundleeq}), the classifying map $f: N\rightarrow BSO(3)$ for the bundle $\xi$ is determined by a pair of maps $(f^\prime, \alpha)\in [S^4, BSO(3)]\times [N, BS^1]$ such that $f=q^\ast(f^\prime)\cdot (Bs)_\ast(\alpha)$ and $\omega_2(\xi)\equiv \alpha~{\rm mod}~2$, where $q$ and $s$ are defined before Lemma \ref{Nbundlelemma}. Moreover, by the discussion after Lemma \ref{Nbundlelemma}, when $\xi$ is non-Spin we suppose that $\alpha$ is primitive.

For the loop homotopy of $M$, we may study $S^1$-bundles over $M$ pullback from those over the $4$-manifold $N$.
Consider the circle bundle
\begin{equation}\label{thes1bnundle}
S^1\stackrel{j}{\longrightarrow} Y\stackrel{\pi}{\longrightarrow} N
\end{equation}
classified by a primitive element $\beta\in H^2(N;\mathbb{Z})$, which defines the simply connected $5$-manifold $Y$. 
Based on the previous remark on the choice of $\alpha$, we make the following convention on the choice of $\beta$:
\begin{itemize}
\item $\beta=\alpha$ if $\xi$ is non-Spin, or
\item $\beta$ can be any primitive element if $\xi$ is Spin.
\end{itemize}

The remaining of this section is devoted the proof of Theorem \ref{decomthm} by dividing it into two cases due to the parity of $\langle \beta^2, [N]\rangle$. In Subsection \ref{subsec: odd}, we first prove Theorem \ref{decomthm} using Lemma \ref{nulllemma} under the assumption that $\langle \beta^2, [N]\rangle$ is odd . This is the case when the circle bundle (\ref{thes1bnundle}) plays an essential role. However, when $\langle \beta^2, [N]\rangle$ is even, we have to apply a different method to prove Theorem \ref{decomthm}. This is done in Subsection \ref{subsec: even}.
\subsection{Case I: $\langle \beta^2, [N]\rangle$ is odd}
\label{subsec: odd}
In this case, by the choice of the circle bundle (\ref{thes1bnundle}), consider the pullback of fibre bundles
\begin{equation}\label{keydiag}
\begin{gathered}
 \xymatrix{
&
S^2 \ar@{=}[r] \ar[d]^{\iota} &
S^2 \ar[d]^{i} \\
S^1\ar@{=}[d] \ar[r]^{\jmath} &
X \ar[d]^{\mathfrak{p}} \ar[r]^{\psi} &
M \ar[d]^{p}\\
S^1 \ar[r]^>>>>{j} &
Y \ar[r]^<<<<{\pi} &
N,
}
\end{gathered}
\end{equation}
which defines the closed $7$-manifold $X$ with bundle projections $\psi$ and $\mathfrak{p}$ onto $M$ and $Y$ respectively.
We show that the induced bundle over $Y$ in Diagram (\ref{keydiag}) is trivial in this case.
\begin{proposition}\label{pdtbundleprop}
If $\langle \beta^2, [N]\rangle$ is odd, then the bundle $\pi^\ast(\xi)$ defined in (\ref{keydiag}) is trivial, and in particular
\[
X\cong S^2\times Y.
\]
\end{proposition}
\begin{proof}
By Lemma \ref{nulllemma}, $\pi^\prime$ is null homotopic.
It implies that $f^\prime \circ q \circ \pi\simeq f^\prime\circ \pi^\prime\circ \mathfrak{q}$ is null homotopic by the homotopy commutativity of the right square in Diagram (\ref{piprimediag}).

If $\xi$ is non-Spin, then $\beta=\alpha$. We obtain the homotopy fibration $Y\stackrel{\pi}{\rightarrow}N\stackrel{\alpha}{\rightarrow} BS^1$, which implies that $\alpha\circ \pi$ is null homotopic, and so is $(Bs_\ast)(\alpha\circ \pi)$. Then by Lemma \ref{pdtbundlelemma} the classifying map $f\circ \pi$ of the bundle $\pi^\ast(\xi)$ is null homotopic, and the proposition follows in this case.

If $\xi$ is Spin, by Lemma \ref{Nbundlelemma} the classifying map $f: N\rightarrow BSO(3)$ of $\xi$ is in the image of $q^\ast$, that is, there exists a map $f^\prime: S^4\rightarrow BSO(3)$ such that $f^\prime\circ q\simeq f$, and then the bundle $\xi$ is the pullback of the bundle $\xi^\prime$ over $S^4$ classified by $f^\prime$. 
Hence $f\circ \pi \simeq f^\prime\circ q \circ \pi$ is null homotopic by the previous argument, and then bundle $\pi^\ast(\xi)$ is trivial. In particular, $X\cong S^2\times Y$ and the proposition follows in this case.

Combining the above two cases, the proposition is proved.
\end{proof}

We can now prove Theorem \ref{decomthm} in the case when $\langle \beta^2, [N]\rangle$ is odd.

\begin{proof}[Proof of Theorem \ref{decomthm} in Case I]
As in the beginning of this subsection, consider the circle bundle $S^1\stackrel{j}{\rightarrow} Y\stackrel{\pi}{\rightarrow} N$ classified by the primitive element $\alpha \in H^2(N;\mathbb{Z})$.
Then by Proposition \ref{pdtbundleprop}, the total space $X$ of the sphere bundle of $\pi^\ast(\xi)$ satisfies $X\cong S^2\times Y$. Hence by Diagram (\ref{keydiag}), we have
\begin{equation}\label{Mdecgeneq}
\Omega M\simeq S^1\times \Omega X\simeq S^1\times \Omega S^2 \times \Omega Y.
\end{equation}
If $d=1$, then $Y$ has to be $S^5$, and hence $\Omega M\simeq S^1\times \Omega S^2\times \Omega S^5$. If $d\geq 2$, by \cite[Example 4.4]{BT2} or \cite{Bas} 
there is a homotopy equivalence
\begin{equation}\label{Ydec1eq}
 \Omega Y\simeq \Omega (S^2\times S^3)\times  \Omega\big(J\vee(J\wedge\Omega (S^2\times S^3))\big)
\end{equation}
with $J=\mathop{\bigvee}\limits_{i=1}^{d-2}(S^2\vee S^3)$. 
Combining (\ref{Mdecgeneq}) with (\ref{Ydec1eq}), we obtain the loop decomposition of $M$ in the theorem. This completes the proof of the theorem when $\langle \beta^2, [N]\rangle$ is odd.
\end{proof}
\subsection{Case II: $\langle \beta^2, [N]\rangle$ is even}
\label{subsec: even}
In this case, the induced bundle $\pi^\ast(\xi)$ defined in (\ref{keydiag}) may not be trivial, and we need to apply a different method to prove Theorem \ref{decomthm}. Indeed, in this case we can work with the sphere bundle $S^2\stackrel{i}{\rightarrow} M\stackrel{p}{\rightarrow}N$ directly, and show that it splits after looping.

\begin{proposition}\label{loopS2MNsplitprop}
If $\langle \beta^2, [N]\rangle$ is even, the sphere bundle $S^2\stackrel{i}{\rightarrow} M\stackrel{p}{\rightarrow}N$ of $\xi$ defined in (\ref{Mdefeq}) is homotopically trivial after looping, and in particular
\[
\Omega M\simeq \Omega S^2\times \Omega N.
\]
\end{proposition}
\begin{proof}
First by Poincar\'{e} duality there exists a class $\alpha\in H^2(N;\mathbb{Z})$ such that $\langle \alpha\cup \beta, [N]\rangle =1$. Since by assumption $\langle \beta^2, [N]\rangle$ is even, $\alpha\neq \beta$. Hence by \cite[Proof of proposition 3.2, Lemma 3.3]{BT1} there exists a Poincar\'{e} duality space $Q$ such that $H^\ast(Q;\mathbb{Z})\cong H^\ast(S^2\times S^2;\mathbb{Z})$ as graded rings, $\Omega Q\simeq \Omega S^2\times \Omega S^2$, and there is a map
\[
h: N\stackrel{}{\longrightarrow} Q,
\]
such that $\Omega h$ has a right homotopy inverse and $h^\ast(x)=\alpha$ with $x\in H^2(Q;\mathbb{Z})$ a generator. Let us fix a homotopy equivalence $e: \Omega S^2\times \Omega S^2\rightarrow \Omega Q$ defined in \cite[Lemma 2.3]{BT1} with its inverse denoted by $e^{-1}$.

Recall $\xi$ is determined by a pair of maps $(f^\prime, \alpha)\in [S^4, BSO(3)]\times [N, BS^1]$. 
By Lemma \ref{Nbundlelemma}, define a rank $3$ vector bundle $\zeta$ over $Q$ by $(f^\prime, x)\in [S^4, BSO(3)]\times [Q, BS^1]$. It follows that $\xi=h^\ast(\zeta)$ and there is a pullback of sphere bundles
\begin{equation}
\label{NHQbundeleq}
\begin{gathered}
 \xymatrix{
S^2 \ar@{=}[d] \ar[r]^{i} &
M \ar[r]^{p} \ar[d]^{\widetilde{h}}&
N\ar[d]^{h}\\
S^2 \ar[r]^{\widetilde{i}} &
\widetilde{Q} \ar[r]^{\widetilde{p}} &
Q ,
}
\end{gathered}
\end{equation}
where the second row is the sphere bundle of $\zeta$ and $\widetilde{h}$ is the induced map.
Since $H^\ast(Q;\mathbb{Z})$ and $H^\ast(S^2;\mathbb{Z})$ are concentrated in even degrees, the Serre spectral sequence for the fibration $S^2\rightarrow \widetilde{Q}\rightarrow Q$ collapses for degree reasons, and then $H^\ast(\widetilde{Q};\mathbb{Z})\cong H^\ast(S^2;\mathbb{Z})\otimes H^\ast(Q;\mathbb{Z})$.
Apply the loop functor to Diagram (\ref{NHQbundeleq}). It is clear that there is a map $i_1\times i_2: S^1\times S^1\rightarrow \Omega\widetilde{Q}$ such that the composition
\[
S^1\times S^1\stackrel{i_1\times i_2}{\longrightarrow} \Omega \widetilde{Q}\stackrel{\Omega \widetilde{p}}{\longrightarrow}\Omega Q \stackrel{e^{-1}}{\longrightarrow} \Omega S^2\times \Omega S^2
\]
is homotopic to $E\times E$ with $E: S^1\rightarrow \Omega S^2$ the suspension map. By the universal property of $\Omega \Sigma$, there is a unique extension $I: \Omega S^2\times \Omega S^2\rightarrow \Omega\widetilde{Q}$ of $i_1\times i_2$ up to homotopy such that 
\[
\Omega S^2\times \Omega S^2\stackrel{I}{\longrightarrow} \Omega \widetilde{Q}\stackrel{\Omega \widetilde{p}}{\longrightarrow}\Omega Q \stackrel{e^{-1}}{\longrightarrow} \Omega S^2\times \Omega S^2
\]
is homotopic to identity. Therefore, the sphere bundle of $\zeta$ splits after looping to give
\[
\Omega \widetilde{Q}\simeq \Omega S^2\times \Omega Q\simeq  \Omega S^2\times\Omega S^2\times \Omega S^2.
\]
In particular, $\Omega \widetilde{i}$ has a left homotopy inverse $\widetilde{r}$, which implies that $\widetilde{r}\circ \Omega\widetilde{h}$ is a left homotopy inverse of $\Omega i$. 
Then the sphere bundle in the top row of Diagram (\ref{NHQbundeleq}) splits after looping, and in particular $\Omega M\simeq \Omega S^2\times \Omega N$. This proves the proposition.
\end{proof}
We can now prove Theorem \ref{decomthm} in the case when $\langle \beta^2, [N]\rangle$ is even.

\begin{proof}[Proof of Theorem \ref{decomthm} in Case II]
Since $\langle \beta^2, [N]\rangle$ is even and $\beta$ is primitive, we have $d\geq 2$.
By Proposition \ref{loopS2MNsplitprop}, $\Omega M\simeq \Omega S^2\times \Omega N$. Further by \cite[Theorem 1.3]{BT1} there is a homotopy equivalence
\[
 \Omega N\simeq S^1\times \Omega (S^2\times S^3)\times  \Omega\big(J\vee(J\wedge\Omega (S^2\times S^3))\big)
\]
with $J=\mathop{\bigvee}\limits_{i=1}^{d-2}(S^2\vee S^3)$. Then in this case the theorem follows by combining the above decompositions.
\end{proof}
\section{The case when $d=0$} 
\label{sec: d=0} 
In this section, we study the case when $d=0$ and prove Theorem \ref{d=0thm} as an immediate corollary of Proposition \ref{Md=0koddprop} and Proposition \ref{Md=0kevenprop}. Indeed, we work in a slightly more general context, that is, to study the loop decomposition of the closed $6$-manifold $M$ with cell structure of the form 
\begin{equation}\label{Md=0eq}
M\simeq S^2\cup e^4\cup e^6.
\end{equation}
Notice that $M$ in Theorem \ref{d=0thm} as the total space of a $S^2$-bundle over $S^4$ is an example of (\ref{Md=0eq}). 
Yamaguchi \cite{Yam} almost determined the homotopy classification of $M$ in (\ref{Md=0eq}) with correction by \cite{MR} and \cite{Bau}, and summarized the criterion whether $M$ has the same homotopy type as a $S^2$-bundle over $S^4$ in \cite[Remark 4.8]{Yam} based on \cite{Sa}.

By (\ref{Md=0eq}) there are generators $x\in H^2(M;\mathbb{Z})$, $y\in H^4(M;\mathbb{Z})$ such that
\begin{equation}\label{xyreleq}
x^2=ky
\end{equation}
for some $k\in \mathbb{Z}$.
Consider the $S^1$-bundle
\begin{equation}\label{Xd=0def}
S^1\stackrel{j}{\longrightarrow} X\stackrel{}{\longrightarrow} M
\end{equation}
classified by $x\in H^2(M;\mathbb{Z})\cong [M, BS^1]$ which defines the closed $7$-manifold $X$. 
Denote $P^n(k)$ be the Moore space such that the reduced cohomology $\widetilde{H}^\ast(P^n(k);\mathbb{Z})\cong \mathbb{Z}/k\mathbb{Z}$ if $\ast=n$ and $0$ otherwise \cite{N}.
\begin{lemma}\label{cellXlemma}
If $k\neq 0$, there is a homotopy equivalence
\[
X\simeq P^4(k)\cup e^7.
\]
\end{lemma}
\begin{proof}
The lemma can be proved directly by analyzing the Serre spectral sequence of the fibration $X\rightarrow M\stackrel{x}{\rightarrow} BS^1$ induced from (\ref{Xd=0def}). Here we provide an alternative proof using results in geometric topology. By \cite[Theorem 1.3]{Jiang}, $X$ is homotopy equivalent to the total space of a $S^3$-bundle over $S^4$. Then by the homotopy classification of $S^3$-bundles over $S^4$ \cite{Sa, CE}, $X$ is homotopy equivalent to $P^4(k^\prime)\cup e^7$ for some $k^\prime\in \mathbb{Z}$. Notice that $\pi_3(X)\cong \pi_3(M)\cong \pi_3(S^2\cup_{k\eta_2}e^4)\cong \mathbb{Z}/k$, where $\eta_2\in\pi_3(S^2)$ is the Hopf element. Then $k=k^\prime$ because $\pi_3(P^4(k^\prime)\cup e^7)\cong \mathbb{Z}/k^\prime$, and the lemma follows.
\end{proof}

Lemma \ref{cellXlemma} has an immediate consequence on the rational homotopy of $M$.\begin{lemma}\label{d=0rationallemma}
Let $M$ be a closed $6$-manifold with cell structure of the form (\ref{Md=0eq}).
Then if $k\neq 0$ there is a rational homotopy equivalence $M\simeq_{\mathbb{Q}}\mathbb{C}P^3$, and if $k=0$ $M\simeq_{\mathbb{Q}} S^2\times S^4$.
\end{lemma}
\begin{proof}
Let $x^2=ky$ for some $k\in \mathbb{Q}$, where $x$ and $y\in H^\ast(M;\mathbb{Q})$ are two generators with ${\rm deg}(x)=2$.
By Poincar\'{e} duality, it is easy to see that the cohomology algebra $H^\ast(M;\mathbb{Q})$ 
is determined by $k$, and is isomorphic to $H^\ast(\mathbb{C}P^3;\mathbb{Q})$ if $k\neq 0$, or $H^\ast(S^2\times S^4;\mathbb{Q})$ if $k=0$. 
Since every simply connected $6$-manifold is formal \cite[Proposition 4.6]{NM}, the rational homotopy type of $M$ is determined by its rational cohomology algebra $H^\ast(M;\mathbb{Q})$. Hence $M\simeq_{\mathbb{Q}}\mathbb{C}P^3$ or $S^2\times S^4$, and the lemma is proved.
\end{proof}

\subsection{The subcase when $k$ is odd} 
\label{subsec: kodd} 
When $k$ is odd, the loop decomposition of the Poincar\'{e} complex $P^4(k)\cup e^7$ was determined by Huang and Theriault \cite{HT}. For any prime $p$, let $S^{m}\{p^r\}$ be the homotopy fibre of the degree $p^r$ map on $S^{m}$.
Let $k=p_1^{r_1}\cdots p_\ell^{r_\ell}$ be the prime decomposition of $k$. By \cite[Theorem 1.1]{HT}, when $k$ is odd there is a homotopy equivalence 
\begin{equation}\label{HTthmeq}
\Omega(P^4(k)\cup e^7)\simeq \prod_{j=1}^{\ell} S^3\{p_j^{r_j}\}\times \Omega S^7.
\end{equation}

\begin{proposition}\label{Md=0koddprop}
Let $M$ be a closed $6$-manifold with cell structure of the form $S^2\cup_{k\eta_2} e^4\cup e^6$. 
If $k$ is odd, then $M$ has the same homotopy type as a $S^2$-bundle over $S^4$, and there is a homotopy equivalence 
\begin{equation}\label{Md=0decomeq}
\Omega M\simeq S^1\times \prod_{j=1}^{\ell} S^3\{p_j^{r_j}\}\times \Omega S^7.
\end{equation}
\end{proposition}
\begin{proof}
The homotopy equivalence (\ref{Md=0decomeq}) follows immediately from Lemma \ref{cellXlemma}, (\ref{Xd=0def}) and (\ref{HTthmeq}). For the first statement, recall that there is the fibre bundle \cite[Section 1.1]{HBJ}
\[
S^2\longrightarrow \mathbb{C}P^3\stackrel{}{\longrightarrow} S^4,
\]
classified by a generator of $\pi_4(BSO(3))\cong\mathbb{Z}$. Pullback this bundle along a self-map of $S^4$ of degree $k$, we obtain the $6$-manifold $M^\prime$ in the following diagram of $S^2$-bundles
\[
\xymatrix{
S^2\ar@{=}[d] \ar[r] &
M^\prime \ar[r]^{} \ar[d] &
S^4\ar[d]^{k} \\
S^2\ar[r] &
\mathbb{C}P^3\ar[r]^{}  &
S^4.
}
\]
It is easy see that $x^{\prime 2}=k y^\prime$ where $x^\prime \in H^2(M^\prime;\mathbb{Z})$ and $y^\prime\in H^4(M^\prime;\mathbb{Z})$ are two generators.
By \cite[Corollary 4.6]{Yam}, when $k$ is odd the homotopy type of $M$ is uniquely determined by $k$, and hence $M\simeq M^\prime$. This completes the proof of the proposition.
\end{proof}

\subsection{The subcase when $k$ is even} 
\label{subsec: keven} 
In \cite[Section 6]{HT}, Huang and Theriault showed that for $P^4(2^r)\cup e^7$ with $r\geq 3$, there is an homotopy equivalence 
\begin{equation}\label{HTpropeq}
\Omega(P^4(2^r)\cup e^7)\simeq S^3\{2^r\}\times \Omega S^7,
\end{equation}
provided there is a map $P^4(2^r)\cup e^7\rightarrow S^4$ inducing a surjection in mod-$2$ homology.

\begin{proposition}\label{Md=0kevenprop}
Let $M$ be a closed $6$-manifold with cell structure of the form $S^2\cup_{2^r\eta_2} e^4\cup e^6$. 
If $r\geq 3$, then there is a homotopy equivalence 
\[
\Omega M \simeq S^1\times S^3\{2^r\}\times \Omega S^7.
\]
\end{proposition}
\begin{proof}
Recall by Lemma \ref{cellXlemma} and its proof $X\simeq P^4(2^r)\cup e^7$, and is homotopy equivalent to the total space of a $S^3$-bundle over $S^4$
\[
S^3\stackrel{}{\longrightarrow} X\stackrel{q}{\longrightarrow} S^4.
\]
It is clear that $q_\ast: H_4(X;\mathbb{Z}/2\mathbb{Z})\rightarrow H_4(S^4;\mathbb{Z}/2\mathbb{Z})$ is surjective. Hence by (\ref{HTpropeq}) $\Omega X\simeq S^3\{2^r\}\times \Omega S^7$. The lemma then follows from (\ref{Xd=0def}) immediately.
\end{proof}

\section{Coformality of $6$-manifolds}
\label{sec: rational}
In this section, we study the rational homotopy theory of $6$-manifolds as an application of our decompositions in Theorem \ref{decomthm}. We briefly recall some necessary terminology used in this section, and for the detailed knowledge of rational homotopy theory one can refer to the standard literature \cite{FHT}.

Recall a $CW$ complex $X$ is rationally {\it formal} if its rational homotopy type is determined by the graded commutative algebra $H^\ast(X;\mathbb{Q})$; and is rationally {\it coformal} if its rational homotopy type is determined by the graded Lie algebra $\pi_\ast(\Omega X)\otimes \mathbb{Q}$, which is called the {\it homotopy Lie algebra} of $X$ denoted by $L_X$. Suppose $(\Lambda V_X, d)$ is a {\it Sullivan model} of $X$. The differential $d=\sum\limits_{i\geq 0}d_i$ with $d_i: V_X\rightarrow \Lambda^{i+1} V_X$, and $(\Lambda V_X, d)$ is {\it minimal} if the linear part $d_0=0$. In the latter case, $V_X$ is dual to $\pi_\ast(\Omega X)\otimes \mathbb{Q}$. Moreover, $X$ is coformal if and only if it has a {\it purely quadratic} Sullivan model $C^\ast (L_X, 0)=(\Lambda (sL_X)^{\#}, d_1)$, where $C^\ast(-)$ is the {\it commutative cochain algebra functor}, $s$ is the suspension and $\#$ is the dual operation.
\begin{proposition}\label{Mcoformalprop}
Let $M$ be a $6$-manifold in Theorem \ref{decomthm} such that $d\geq 2$. Then $M$ is coformal.
\end{proposition}
\begin{proof}
Consider the $S^2$-bundle 
\begin{equation}\label{s2fibqeq}
S^2\stackrel{i}{\rightarrow} M\stackrel{p}{\rightarrow} N
\end{equation}
in Diagram (\ref{keydiag}). By \cite[Proposition 4.4]{NM} $N$ is coformal since $d\geq 2$, and hence has a minimal Sullivan model of the form $C^\ast (L_N, 0)=(\Lambda (sL_N)^{\#}, d_1)$ as the associated commutative cochain algebra of $(L_N,0)$ \cite[Example 7 in Chapter 24 (f)]{FHT}.
Let $\widehat{p}: C^\ast (L_N, 0)\stackrel{}{\rightarrow} (C^\ast (L_N)\otimes \Lambda(a, b), d)$ be a relative minimal Sullivan model of $p$, whose quotient $(\Lambda(a, b),\bar{d})$ is a minimal Sullivan model of $S^2$ with $db=a^2$ and ${\rm deg}(a)=2$.
It follows that there is the short exact sequence of the linear part of the model of (\ref{s2fibqeq})
\begin{equation}\label{lineareq}
0\rightarrow ((sL_N)^{\#}, 0)\rightarrow ((sL_N)^{\#}\oplus \mathbb{Q}(a,b), d_0) \stackrel{}{\rightarrow} (\mathbb{Q}(a,b),0)\rightarrow 0,
\end{equation}
such that $H^\ast((sL_N)^{\#}\oplus \mathbb{Q}(a,b), d_0)$ is dual to $\pi_\ast(M)\otimes \mathbb{Q}$. However, since the homotopy groups of (\ref{s2fibqeq}) splits by Theorem \ref{decomthm} and its proof, we see from (\ref{lineareq}) that the linear part $d_0=0$ for $(sL_N)^{\#}\oplus \mathbb{Q}(a,b)$ and hence $(C^\ast (L_N)\otimes \Lambda(a, b), d)$ is a minimal model of $M$. 

To show $M$ is coformal, it suffices to show that the differential $d$ is quadratic on $\mathbb{Q}(a,b)$ in $(C^\ast (L_N)\otimes \Lambda(a, b), d)$. Since $N$ is simply connected, $(sL_N)^{\#}$ concentrates in degrees larger than or equal to $2$. It follows that by the minimality of $(C^\ast (L_N)\otimes \Lambda(a, b), d)$ and degree reasons
\[
da=0, \ \ db=a^2+ay+\sum\limits_i z_i w_i,
\]
for some degree $2$ elements $y$, $z_i$ and $w_i\in (sL_N)^{\#}$. Hence $d=d_1$ in $(C^\ast (L_N)\otimes \Lambda(a, b), d)$. This shows that $M$ is coformal and the proposition is proved.
\end{proof}

We can now prove Theorem \ref{coformalthm}.
\begin{proof}[Proof of Theorem \ref{coformalthm}]
First it is well known that $\mathbb{C}P^i$ is not coformal for $i\geq 2$ by \cite[Example 4.7]{NM}. If $d=1$, then $M$ is determined by a fibre bundle $S^2\stackrel{}{\rightarrow} M\stackrel{}{\rightarrow} \mathbb{C}P^2$. It has a model of the form
\[
(\Lambda(c, x), dx=c^3)\stackrel{}{\longrightarrow} (\Lambda(c, x, a, b), \tilde{d})\stackrel{}{\longrightarrow}(\Lambda(a, b), db=a^2),
\]
where ${\rm deg}(c)={\rm deg}(a)=2$. By degree reason $\tilde{d}(a)=0$, and $\tilde{d}(b)=a^2+kc^2$ for some $k\in \mathbb{Q}$, which implies that $(\Lambda(c, x, a, b), \tilde{d})$ is minimal. However, $\tilde{d}$ is not quadratic as $\tilde{d}(x)=c^3$. Hence $M$ is not coformal. 

When $d\geq 2$, by Proposition \ref{Mcoformalprop} $M$ is coformal. Moreover, Neisendorfer and Miller \cite[Proposition 4.6]{NM} showed that every simply connected $6$-manifold is formal. Hence by \cite[Theorem 1.2]{Ber} $M$ is Koszul. By \cite[Theorem 1.3]{Ber} there is an isomorphism of graded Lie algebras
\[
\pi_\ast(\Omega M)\otimes\mathbb{Q}\cong H^\ast(M;\mathbb{Q})^{!\mathscr{L}ie},
\]
where $(-)^{!\mathscr{L}ie}$ is the Koszul dual Lie functor.
\end{proof}

\bibliographystyle{amsalpha}

\end{document}